\documentclass[12pt,reqno,twoside]{amsart}

\usepackage[a4paper, top=4cm, bottom=3cm, left=3cm, right=3cm]{geometry}
\usepackage{amsmath,amssymb,amsthm,mathtools, manfnt}

\newtheorem{thm}{Theorem}[section]
\newtheorem{lem}[thm]{Lemma}
\newtheorem{prop}[thm]{Proposition}
\theoremstyle{definition}
\newtheorem*{thm*}{Theorem}

\newcommand{\legendre}[2]{\left(\frac{#1}{#2} \right)}
\newcommand{\Q}{\mathbb{Q}}
\newcommand{\E}{\mathbb{E}}
\newcommand{\R}{\mathbb{R}}
\newcommand{\C}{\mathbb{C}}
\renewcommand{\P}{\mathbb{P}}
\DeclareMathOperator{\Res}{Res}
\DeclareMathOperator{\rad}{rad}
\newcommand{\cube}{\mbox{\mancube}}

\numberwithin{equation}{section}

\title[Central values of Hecke $L$-functions of $n$-th order characters]{Distribution of central values of Hecke $L$-functions of $n$-th order characters}
\author{Lewis McGoldrick}
\address{School of Mathematics, Georgia Institute of Technology, Atlanta, USA}
\email{lmcgoldrick6@gatech.edu}

\begin{document}

\maketitle

\begin{abstract}
We prove a one-sided central limit theorem for $\log |L(\tfrac 12,\chi)|$ in a thin family of $n$-th order Hecke characters, for $n \geq 3$. The proof combines the method of Radziwiłł and Soundararajan with a recent asymptotic formula of Diaconu, Ion, Paşol, and Popa for the twisted second moment with a power-saving error term. We also obtain upper bounds of the conjectured order of magnitude for the $2k$-th moments for $0< k< 1$.
\end{abstract}

\section{Introduction}
A fundamental theorem due to Selberg \cite{selberg} shows that the distribution of $\log |\zeta(\frac 12+it)|$ is approximately described by a Gaussian with mean 0 and variance $\frac 12 \log\log |t|$. More recently, Radziwiłł and Soundararajan \cite{radziwill-soundararajan} gave a simpler proof, inspired by their work \cite{maks-sound} on the analogous problem for families of $L$-functions. 
\begin{thm*}[\cite{radziwill-soundararajan}, Theorem 1]
Let $V$ be a fixed positive real number. Then as $T \to \infty$, uniformly for all $v \in [-V,V]$,
\[
\frac 1T \operatorname{meas} \left\{ T \leq t \leq 2T \ : \ \frac{\log |\zeta(\tfrac 12 + it)|}{\sqrt{\tfrac 12 \log\log T}} \geq v  \right\} \sim \frac{1}{\sqrt{2\pi}} \int_v^\infty e^{-u^2/2}du.
\]
\end{thm*}
In the setting of families, Katz and Sarnak's conjectural description of the distribution of low-lying zeros in families of $L$-functions \cite{katz-sarnak} led Keating and Snaith \cite{keating-snaith} to conjecture analogues of Selberg's theorem for the central values in certain families of $L$-functions. One can view Selberg's theorem as a statement about the distribution of central values of the family of $L$-functions $\zeta(s+it)$ obtained by varying the archimedean parameter $t$. In contrast, Keating and Snaith's conjectures concern the analogous distribution obtained from varying the non-archimedean parameter, e.g. a Dirichlet character.

Unlike Selberg's theorem, the full Keating-Snaith conjectures remain out of reach. In his paper on the moments of the Riemann zeta function \cite{soundararajan-moments}, Soundararajan observed that his methods, assuming the Generalized Riemann Hypothesis (GRH), would yield a one-sided version of Keating and Snaith's conjecture. Subsequently, Hough obtained an unconditional one-sided result for an orthogonal and a symplectic family of $L$-functions. For example, let $s(D)$ be the set of odd, square-free $d$ bounded by $\tfrac{D}{2}<d \leq D$, and set $\chi_d(n)=\legendre dn$ the Kronecker symbol, with corresponding Dirichlet $L$-function $L(s,\chi_d)$. Then the second part of Hough's Corollary 1.1 gives the following upper-tail bound.
\begin{thm*}[\cite{hough}, Corollary 1.1]
For fixed $A$ as $D \to \infty$,
\[
\P\left[ d \in s(D)  : \frac{\log \left| L\left( \tfrac 12, \chi_{8d} \right) \right| - \tfrac 12 \log\log D }{\sqrt{\log\log D}} > A \right] \leq \frac{1}{\sqrt{2\pi}}\int_A^\infty e^{-\frac{x^2}{2}}dx+o_A(1).
\]
\end{thm*}

His proof follows Selberg's original method, and thus relies on zero-density estimates. An alternative method, avoiding such estimates, was provided by Radziwiłł and Soundararajan in \cite{maks-sound}. They describe a robust method that works for a wide range of families, provided that one can prove asymptotics for the twisted first moments with a power-saving error term. They also give an interpolation argument showing that such asymptotics for a single moment imply upper bounds of the expected size for all smaller moments.

In a recent work on non-vanishing for cubic Hecke $L$-functions, David, de Faveri, Dunn, and Stucky \cite{david-defaveri-dunn-stucky} obtained an asymptotic for the second moment with a power-saving error term for the thin family of cubic characters coming from the power residue symbol for square-free modulus. Following this, a similar result for quartic characters was proved by Castillo, de Faveri, and Dunn in \cite{castillo-defaveri-dunn}. The generalization to higher order characters was recently given by Diaconu, Ion, Paşol, and Popa in  \cite{diaconu-ion-pasol-popa}, where they obtained a general twisted second moment asymptotic for the thin family of $n$-th order characters for all $n \geq 3$ with a power-saving error term. Using this second moment asymptotic, along with the method of Radziwiłł and  Soundararajan, we prove a one-sided version of the Keating-Snaith analogue of Selberg's theorem for the thin family $\mathcal F_n$ of $n$-th order characters for $n \geq 3$. We also obtain upper bounds of the expected size for the $2k$-th moments, $0 <k < 1$.

Throughout the paper, we fix the order $n \geq 3$, a number field $K$ containing the $2n$-th roots of unity, and a finite set of prime ideals $S$, along with some fixed character data described in the following section. We denote by $I(S)$ the integral ideals coprime to every prime in $S$, and think of ideals $\mathfrak q \in I(S)$ as indices for Hecke characters $\chi_\mathfrak q$ that are constructed explicitly in Section 2.1, and let $\mathcal F_n=\{\mathfrak q \in I(S) : \mathfrak q \text{ is square-free}\}$. The family of $L$-functions attached to characters $\chi_\mathfrak q$ for $\mathfrak q \in \mathcal F_n$ is unitary, and the distribution of $\log |L(\tfrac 12, \chi_\mathfrak q)|$ is predicted to be Gaussian with mean $0$ and variance $\tfrac 12 \log\log N(\mathfrak q)$, where we assume the convention $\log 0 = -\infty$. 
\begin{thm}
\label{clt}
Let $V$ be a fixed real number. Then as $X \to \infty$,
\begin{multline*}
\left| \left\{\mathfrak q \in \mathcal F_n \ : \ N(\mathfrak q) \leq X, \frac{\log |L(\frac 12, \chi_\mathfrak q)|}{\sqrt{\tfrac 12\log\log X}} \geq V \right\} \right|
\\
\leq \left( \frac{1}{\sqrt{2\pi}}\int_{V}^\infty e^{-\frac{x^2}{2}}dx +o(1) \right) | \{\mathfrak q \in \mathcal F_n \ : \  N(\mathfrak q) \leq X \} |.
\end{multline*}
\end{thm}
\begin{thm}
\label{moments}
Let $0 < k < 1$. Then as $X \to \infty$,
\[
\sum_{\substack{\mathfrak q \in \mathcal F_n \\ N(\mathfrak q) \leq X}} |L(\tfrac{1}{2}, \chi_\mathfrak q)|^{2k} \ll_{k} X(\log X)^{k^2}.
\]
\end{thm}
Analogous lower bounds to Theorem \ref{clt} appear to be substantially more difficult. However, recently Radziwiłł and Soundararajan developed a technique complementary to that of \cite{maks-sound} to obtain lower bounds conditional on GRH \cite{radziwill-soundararajan-conditional}. This has already essentially been adapted to the family $\mathcal F_3$ for $K=\Q(\omega)$ by Lin and Wong in \cite{lin-wong}, where they extended results on the 1-level density of zeros by David and Güloğlu \cite{david-guloglu}.	
\subsection*{Acknowledgements}
The author thanks Alexander Dunn for suggesting the project, and for his guidance and encouragement.
\section{Preliminaries}
\subsection{Character construction}
We follow the definitions in \cite{friedberg-hoffstein-lieman} and \cite{diaconu-ion-pasol-popa} to construct a family of $n$-th order Hecke $L$-functions. 

Let $K$ be a number field containing the $2n$-th roots of unity, with ring of integers $\mathcal O$. We shall fix the following data throughout the whole paper, which we define in this section: $S$, $\mathfrak c$, $\mathcal E$, $\mathcal E_0$, $m_{E_0}$. Let $S$ be a finite set of primes of $K$, containing all primes ramified in $K/\Q$, and such that the localization $\mathcal O_{S}$ has class number one. Then we define $I(S)$ to be the set of integral ideals coprime to every prime in $S$. We think of an element $\mathfrak q \in I(S)$ as an index for a character $\chi_\mathfrak q$ which we shall now define.

We start with the $n$-th power residue symbol $\legendre{.}{.}$, defined locally for an element $q \in K^\times$ coprime to a prime $\mathfrak p\not\in S$ to be the $n$-th root of unity satisfying
\[
\legendre{q}{\mathfrak p} \equiv q^{(N(\mathfrak p)-1)/n} \bmod \mathfrak p,
\]
and extended multiplicatively to an arbitrary integral ideal $\mathfrak n\in I(S)$ coprime to $q$ by 
\[
\legendre{q}{\mathfrak n} = \prod_{\mathfrak p \mid \mathfrak n} \legendre{q}{\mathfrak p}^{v_\mathfrak p(\mathfrak n)}.
\]
To define $\chi_\mathfrak q(\mathfrak n)$ we must make a consistent choice of an element $q$. First, fix an ideal $\mathfrak c = \prod_{\mathfrak p \in S} \mathfrak p^{\nu_\mathfrak p},$ where $\nu_\mathfrak p$ is chosen large enough so that $q \equiv 1 \bmod \mathfrak c$ implies $q \in (K_\mathfrak p^\times)^n$ for all $\mathfrak p \in S$. Consider the group $R_\mathfrak c = H_\mathfrak c/(H_\mathfrak c)^n$, where $H_\mathfrak c$ is the ray class group modulo $\mathfrak c$. This group is a product of cyclic groups $R_\mathfrak c \cong C_1 \times C_2 \times \dots$. For each cyclic group $C_i$, we choose one generator $[E_0]$, $E_0 \in I(S)$, and denote by $\mathcal E_0$ the set of these $E_0$. Now we obtain a full set of representatives $\mathcal E$ of $R_\mathfrak c$ by taking elements of the form $E = \prod_{E_0 \in \mathcal E_0} E_0^{v_{E_0}}$.

Since $\mathcal O_S$ has class number one, for each representative $E_0 \in \mathcal E_0$ we pick an element $m_{E_0} \in \mathcal O$ such that $m_{E_0} \mathcal O_S = E_0 \mathcal O_S$. We extend this to any $E \in \mathcal E$ by setting $m_E = \prod_{E_0 \in \mathcal E_0} m_{E_0}^{v_{E_0}}$.

Now given a $\mathfrak q$, we pick a corresponding $E \in \mathcal E$ so that $[\mathfrak q] = [E]$ in $R_\mathfrak c$. Consequently, we can write $\mathfrak q E^{-1} =\mathfrak g^n$ in $H_\mathfrak c$ for some $\mathfrak g \in I(S)$, $(\mathfrak g, \mathfrak n)=1$. Therefore, there is some generator $q \equiv 1 \bmod \mathfrak c$ such that $\mathfrak q E^{-1} \mathfrak g^{-n} = (q)$. Then we write $\mathfrak q = (q) E \mathfrak g^n$, and define the character
\[
\chi_\mathfrak q(\mathfrak n) = \legendre{qm_E}{\mathfrak n},
\]
Changing our choice of $\mathfrak g$ does not affect the above value since it is raised to the power of $n$. Lastly, set $\chi_\mathfrak q(\mathfrak n)=0$ if $(\mathfrak q, \mathfrak n)\neq 1$ or $(\mathfrak n, S) \neq 1$.

We define the thin family $\mathcal F_n$ to be the set of indices $\mathfrak q \in I(S)$ that are square-free. We shall sometimes abuse notation and refer to elements of $\mathcal F_n$ as the characters $\chi_\mathfrak q$ instead.

The character $\chi_\mathfrak q$ is a Hecke character modulo $\mathfrak c \mathfrak q$ of trivial infinity type, and its conductor is the ideal $\mathfrak c_\mathfrak q \mathfrak q_0$, where $\mathfrak c_\mathfrak q \mid \mathfrak c$ is some ideal dividing $\mathfrak c$, and $\mathfrak q_0$ is the product of the prime ideals dividing the $n$-th power-free part of $\mathfrak q$ \cite{goldmakher-louvel}. Consequently, for the thin family $\mathcal F_n$, ordering characters by conductor is essentially the same as ordering them by the ideal norm of $\mathfrak q$. This is the main reason we consider the thin family, as opposed to allowing non-squarefree $\mathfrak q$.

Now we recall the reciprocity law proved in \cite{diaconu-ion-pasol-popa}.
\begin{lem}[\cite{diaconu-ion-pasol-popa}, Lemma 3.1.5.]
\label{reciprocitylem}
Let $\mathfrak a, \mathfrak b$ be coprime $n$-th power-free ideals also coprime to $S$. Let $E = [\mathfrak b]$ and $G = [\mathfrak a]$. Then
\begin{align}
\label{reciprocity}
\chi_\mathfrak b(\mathfrak a) = \chi_\mathfrak a(\mathfrak b) (m_G, m_E)_S,
\end{align}
where $(.,.)_S = \prod_{\mathfrak p \in S} (.,.)_\mathfrak p$ is the product of $n$-th local Hilbert symbols.
\end{lem}
To a character $\chi_\mathfrak q$ we attach for $\Re s>1$ the Hecke $L$-function
\[
L(s,\chi_\mathfrak q) := \sum_{\mathfrak n\neq 0} \frac{\chi_\mathfrak q(\mathfrak n)}{N(\mathfrak n)^s},
\]
which extends to an entire function if $\chi_\mathfrak q$ is non-trivial. We note that this is not primitive, since we have removed the Euler factors at the bad primes in $S$.
\subsection{Probability measures}
We shall denote by $\P_{\leq X}$ the uniform probability measure defined for $E \subset \{ \mathfrak q \in \mathcal F_n \ : \  N(\mathfrak q) \leq X\}$ by
\[
\P_{\leq X}(E) = \frac{\sum_{\mathfrak q \in E}1}{\sum_{\mathfrak q \in \mathcal F_n, N(\mathfrak q) \leq X} 1}.
\]
We also define a smoothed analogue. Let $0 \leq \Phi\leq 1$ be a smooth approximation to the indicator function, supported on $[\tfrac 12,\tfrac 52]$ and equal to one on $[1,2]$. Then for $X>0$ and $E \subset \mathcal F_n$ we define
\[
\P_X^\Phi(E) := \frac{\sum_{\mathfrak q \in E} \Phi\left(\frac{N(\mathfrak q)}{X}\right)}{\sum_{\mathfrak q \in \mathcal F_n} \Phi \left( \frac{N(\mathfrak q)}{X}\right)}.
\]
Also, let $\E_{\leq X}$ and $\E_X^\Phi$ denote the expectations with respect to these probability measures.

We can now rewrite the statement of Theorem \ref{clt}:

Let $V$ be a fixed real number. Then as $X \to \infty$,
\begin{align}
\label{updatedclt}
\P_{\leq X} \left( \frac{\log |L(\frac 12, \chi_\mathfrak q)|}{\sqrt{\frac 12 \log\log X}} \geq V \right) \leq \frac{1}{\sqrt{2\pi}} \int_V^\infty e^{-\frac{x^2}{2}}dx+o(1).
\end{align}
We also have a smoothed version of Theorem \ref{moments}.
\begin{thm}
\label{smoothedmoments}
Let $0 < k < 1$. Then as $X \to \infty$,
\[
\E_X^\Phi\left[ | L(\tfrac 12, \chi_\mathfrak q)|^{2k} \right] \ll_{k,\Phi} (\log X)^{k^2}.
\]
\end{thm}
We note that Theorem \ref{moments} follows immediately from the smoothed version after a dyadic decomposition.

As one might expect comparing the smoothed and unsmoothed theorems, the family $\mathcal F_n$ is of size $X$ after restricting to $\mathfrak q$ with $N(\mathfrak q)\leq X$. Indeed, we have the following estimate.
\begin{lem}
\label{familysizelem}
Let $\widehat{\Phi}(s) = \int_0^\infty \Phi(x) x^{s-1} dx$ denote the Mellin transform of $\Phi$. Then
\begin{align}
\label{familysize}
\sum_{\mathfrak q \in \mathcal F_n} \Phi \left( \frac{N(\mathfrak q)}{X}\right) = \prod_{\mathfrak p \in S} \left( 1 + \frac 1{N(\mathfrak p)} \right)^{-1} \frac{\widehat{\Phi}(1)\Res_{s=1}\zeta_K(s)}{\zeta_K(2)} X + O(X^{1/2+\varepsilon}).
\end{align}
\begin{proof}
By Mellin inversion,
\[
\sum_{\mathfrak q \in \mathcal F_n} \Phi\left(\frac{N(\mathfrak q)}{X} \right) = \frac{1}{2\pi i} \int_{(2)} \widehat{\Phi}(s) D(s) X^s ds,
\]
where
\[
D(s) = \prod_{\mathfrak p \in S} \left( 1 + \frac{1}{N(\mathfrak p)^s}\right)^{-1} \frac{\zeta_K(s)}{\zeta_K(2s)}.
\]
Moving the contour to the line $\Re s = \tfrac 12+\varepsilon$ we pick up a residue from passing a simple pole of $\zeta_K$ at one, corresponding to the main term in \eqref{familysize}. We bound the remaining integral using the convexity bound and rapid decay of $\widehat \Phi$, obtaining
\[
\frac{1}{2\pi i} \int_{(\tfrac 12+\varepsilon)} \widehat{\Phi}(s) D(s) X^s ds \ll X^{1/2+\varepsilon}.
\]
\end{proof}
\end{lem}
We now give proofs of the main theorems.
\section{Proofs of Theorems}
In this section, we condense the proof of the main theorems into four propositions.
\subsection{Theorem \ref{clt}}
\begin{proof}
We follow the method of Radziwiłł and Soundararajan as in \cite{maks-sound}, first approximating $\log |L(\frac 12, \chi_\mathfrak q)|$ with a short Dirichlet polynomial:
\[
\mathcal P(\mathfrak q):= \sum_{\mathfrak p \in P} \frac{\chi_\mathfrak q(\mathfrak p)}{N(\mathfrak p)^{1/2}}, \ \ \text{where } P := \{ \mathfrak p \not\in S\ : \ N(\mathfrak p) \leq X^{1/(\log\log X)^2}\}.
\]
We now show that $\Re \mathcal P(\mathfrak q)$ is asymptotically normally distributed using the method of moments. We obtain the following result on the moments of $\Re \mathcal P(\mathfrak q)$.
\begin{prop}
\label{pmoments}
Let $k \geq 1$ be a positive integer. Then for large $X$,
\[
\E_X^\Phi\left[ \left( \frac{\Re\mathcal P(\mathfrak q)}{\sqrt{ \frac 12 \log\log X}}\right)^k \right] = M_k+o(1),
\]
where 
\[
M_k := \begin{cases} (k-1)!!, & \text{ if } k \text{ is even};
\\
0, & \text{ if } k \text{ is odd}.
\end{cases}
\]

\end{prop}
Since the normal distribution is determined by its moments, by the method of moments, we find
\[
\P_X^\Phi \left( \frac{\Re \mathcal P(\mathfrak q)}{\sqrt{\frac 12 \log\log X}} \geq V \right) = \frac{1}{\sqrt{2\pi}} \int_V^\infty e^{-\frac{x^2}{2}}dx + o(1).
\]
To unsmooth this distribution result, fix $\varepsilon>0$. Then take $\Phi=\Phi_\varepsilon$ to be supported on $[1-\varepsilon,2+\varepsilon]$. Note that we have
\[
\sum_{\substack{\mathfrak q \in \mathcal F_n \\ X \leq N(\mathfrak q) \leq 2X}} 1 = \Res_{s=1} D(s) X+o(X),
\]
as opposed to
\[
\sum_{\mathfrak q \in \mathcal F_n} \Phi \left( \frac{N(\mathfrak q)}{X}\right) = \widehat{\Phi}(1) \Res_{s=1} D(s) X+O(X^{1/2+\varepsilon}).
\]
By definition, $\widehat{\Phi}(1) = 1 + O(\varepsilon)$, and so taking $\varepsilon \to 0$ with a dyadic decomposition gives
\begin{align}
\label{pdistribution}
\P_{\leq X} \left( \frac{\Re \mathcal P(\mathfrak q)}{\sqrt{ \frac 12 \log\log X}} \geq V \right) = \frac{1}{\sqrt{2\pi}} \int_V^\infty e^{-\frac{x^2}2}dx+o(1).
\end{align}
Therefore, to prove Theorem \ref{clt}, it suffices to show that the number of $\mathfrak q$ satisfying 
\begin{align}
\label{condition1}
\frac{\log |L(\frac 12 , \chi_\mathfrak q)|}{\sqrt{\frac 12 \log\log X}} \geq V \text{ and } \frac{\Re \mathcal P(\mathfrak q)}{\sqrt{\frac 12 \log\log X}}<V,
\end{align}
is negligible, of size $o(X)$.

Now fix $\varepsilon>0$. We further split into two cases; one where $\Re \mathcal P(\mathfrak q)$ is approximately $\log |L(\frac 12, \chi_\mathfrak q)|$:
\begin{align}
\label{condition2}
\frac{\Re \mathcal P(\mathfrak q)}{\sqrt{ \frac 12 \log\log X}} \geq V-\varepsilon,
\end{align}
and one where $\Re \mathcal P(\mathfrak q)$ is substantially far away from $\log |L(\frac 12, \chi_\mathfrak q)|$:
\begin{align}
\label{condition3}
\frac{\Re \mathcal P(\mathfrak q)}{\sqrt{\frac 12 \log\log X}}< V-\varepsilon.
\end{align}
By (\ref{pdistribution}) the number of $\mathfrak q$ satisfying both (\ref{condition1}) and (\ref{condition2}) is bounded by $O(\varepsilon X)+o_\varepsilon(X)$.

Any $\mathfrak q$ satisfying (\ref{condition1}) and (\ref{condition3}) must have
\[
\log |L(\tfrac 12, \chi_\mathfrak q)| - \Re \mathcal P(\mathfrak q) > \varepsilon \sqrt{\tfrac 12\log\log X}.
\]
We bound this using Markov's inequality for the smoothed version via a dyadic decomposition, so the number of such $\mathfrak q$ is 
\begin{align}
\label{markov}
\ll X\E_X^\Phi \left[ |L(\tfrac 12, \chi_\mathfrak q)|^2 e^{-2\Re \mathcal P(\mathfrak q)} \right] e^{-2\varepsilon \sqrt{\frac 12 \log\log X}}.
\end{align}
Therefore, it suffices to obtain a good bound for the above expectation.

It will be technically convenient to work with a truncated exponential $E_\ell(x) = \sum_{j=0}^\ell \frac{x^j}{j!}$ instead. Given that our primes have norm bounded by $X^{1/(\log\log X)^2}$, we truncate up to $\ell \asymp \log\log X$ so that the polynomial has length $X^{o(1)}$. Also, for our truncated exponential to give a good approximation, we must work in the range $|\Re \mathcal P(\mathfrak q)| \ll \ell$. Therefore, we deal with the $\mathfrak q$ satisfying $\Re \mathcal P(\mathfrak q) \leq - \log\log X$ separately. Indeed, Proposition \ref{pmoments} shows that the number of such $\mathfrak q$ is $o(X)$.

For the final remaining $\mathfrak q$, we use a precise upper bound for $e^x$ (\cite{maks-sound}, Lemma 1), which yields
\[
|L(\tfrac{1}{2},\chi_\mathfrak q)|^2 e^{-2\Re \mathcal P(\mathfrak q)} \ll |L(\tfrac 12,\chi_\mathfrak q)|^2 E_\ell(-2\Re \mathcal P(\mathfrak q)),
\]
where $\ell = 20 \lceil \log\log X \rceil$. The constant 20 is required for our application of Rankin's trick in Section 5. 

We shall prove the following proposition:
\begin{prop}
\label{kequals-1}
Let $\ell$ be as above. Then,
\[
\E_X^\Phi\left[ |L(\tfrac{1}{2},\chi_\mathfrak q)|^2 E_\ell(-2\Re \mathcal P(\mathfrak q)) \right] \ll (\log\log X)^2.
\]
\end{prop}
Together with \eqref{markov} this shows that the number of our final remaining $\mathfrak q$ satisfying (\ref{condition1}) is $o(X)$. After letting $X \to \infty$ we take $\varepsilon \to 0$ to obtain Theorem \ref{clt}.
\end{proof}
\subsection{Theorem \ref{moments}}
\begin{proof}
In \cite{maks-sound}, Radziwiłł and Soundararajan make use of Young's inequality to bound the $k$-th power of the central value in terms of a short Dirichlet sum and the first power of the central value. We recall their key lemma.
\begin{lem}[\cite{maks-sound}, Lemma 2]
\label{youngbound}
Let $y \geq 0$ be a real number. Suppose that $x_1,\dots,x_R$ are real numbers, and $\ell_1,\dots, \ell_R$ are positive even integers. Then for any $0 < k < 1$ we have
\begin{multline*}
y^k \leq Cky \prod_{j=1}^R E_{\ell_j}((k-1)x_j)+C(1-k) \prod_{j=1}^R E_{\ell_j}(kx_j) 
\\
+ \sum_{r=0}^{R-1} \left( Cky \prod_{j=1}^r E_{\ell_j}((k-1)x_j)+C(1-k) \prod_{j=1}^r E_{\ell_j} (kx_j)\right) \left( \frac{e^2 x_{r+1}}{\ell_{r+1}}\right)^{\ell_{r+1}},
\end{multline*}
where $C = \exp((e^{-\ell_1}+\dots+e^{-\ell_R})/16)$.
\end{lem}
Following Radziwiłł and Soundararajan's notation, for some fixed sets of primes $P_1,\dots,P_R$, and even integers $\ell_1,\dots,\ell_R$, we define
\[
\mathcal P_j(\mathfrak q):= \sum_{\mathfrak p \in P_j} \frac{\chi_\mathfrak q(\mathfrak p)}{N(\mathfrak p)^{1/2}}
\]
and
\[
\mathcal A_j(\mathfrak q):= E_{\ell_j}(2(k-1) \Re \mathcal P_j(\mathfrak q)), \ \mathcal B_j(\mathfrak q) := E_{\ell_j}(2k\Re \mathcal P_j(\mathfrak q)).
\]
Then after applying Lemma \ref{youngbound}, we obtain
\begin{multline}
\label{keylemma}
|L(\tfrac{1}{2},\chi_\mathfrak q)|^{2k} 
\\
\leq Ck|L(\tfrac{1}{2},\chi_\mathfrak q)|^2 \left( \prod_{j=1}^R \mathcal A_j(\mathfrak q) + \sum_{r=0}^{R-1} \prod_{j=1}^r \mathcal A_j(\mathfrak q) \left( \frac{ 2e^2 \Re \mathcal P_{r+1}(\mathfrak q)}{\ell_{r+1}}\right)^{\ell_{r+1}} \right)
\\
+ C(1-k)\left( \prod_{j=1}^R \mathcal B_j(\mathfrak q) + \sum_{r=0}^{R-1} \prod_{j=1}^r \mathcal B_j(\mathfrak q) \left( \frac{ 2e^2 \Re \mathcal P_{r+1}(\mathfrak q)}{\ell_{r+1}}\right)^{\ell_{r+1}}\right).
\end{multline}
This essentially mirrors the setup of Proposition \ref{kequals-1}, but we have now split the primes into $R$ disjoint sets, which gives us better control of the size of the Dirichlet sums. To keep the Dirichlet sums short enough, given $X>0$, we define
\[
P_1 := \{\mathfrak p \not\in S \ : \ N(\mathfrak p) \leq X^{1/\ell_1^2}\}, \ P_j := \{ \mathfrak p \not\in S \; : \; X^{1/\ell_{j-1}^2}<N(\mathfrak p) \leq X^{1/\ell_j^2} \}
\]
and
\[
\ell_1:=5400 \lceil \log\log X \rceil, \ \ell_j = 5400 \lceil \log \ell_{j-1} \rceil,
\]
where $R$ is the largest integer such that $\ell_R>10^{6n}$, so that effectively we are summing over all primes up to $X^{1/\ell_R^2}<X^{1/10^{12n}}$ (we note this is much larger than it needs to be for the proof to work).

We note that $\sum_{j=1}^R \frac{1}{\ell_j} < \frac{1}{10^{5n}}$, which will ensure that the Dirichlet sums in the proofs of Propositions \ref{bound1} and \ref{bound2} are sufficiently short. Let $S_j = \sum_{\mathfrak p \in P_j} N(\mathfrak p)^{-1}$. Then an application of Mertens' theorem yields
\[
S_j = 2 \log \ell_{j-1} - 2\log \ell_j +o(1) \leq \ell_j/2700 \ \text{ for } \ j \geq 2,
\]
and $S_1 \leq \log\log X+O(1)$.

We split the problem of bounding the average of (\ref{keylemma}) into two propositions.
\begin{prop}
\label{bound1}
\[
\E_X^\Phi \left[ \left( \prod_{j=1}^R \mathcal B_j(\mathfrak q) + \sum_{r=0}^{R-1} \prod_{j=1}^r \mathcal B_j(\mathfrak q) \left( \frac{ 2e^2 \Re \mathcal P_{r+1}(\mathfrak q)}{\ell_{r+1}}\right)^{\ell_{r+1}}\right)\right] \ll (\log X)^{k^2}.
\]
\end{prop}
\begin{prop}
\label{bound2}
\[
\E_X^\Phi \left[|L(\tfrac{1}{2},\chi_\mathfrak q)|^2 \left( \prod_{j=1}^R \mathcal A_j(\mathfrak q) + \sum_{r=0}^{R-1} \prod_{j=1}^r \mathcal A_j(\mathfrak q) \left( \frac{2e^2 \Re \mathcal P_{r+1}(\mathfrak q)}{\ell_{r+1}}\right)^{\ell_{r+1}}\right)\right] \ll (\log X)^{k^2}.
\]
\end{prop}
Using the above two propositions and summing dyadically we immediately obtain Theorem \ref{moments}.
\end{proof}
\section{Proof of Proposition \ref{pmoments}}
\begin{proof}
Inserting the definition of $\mathcal P(\mathfrak q)$ gives
\[
\E_X^\Phi\left[ (\Re \mathcal P(\mathfrak q))^k \right]=\frac{1}{2^k}\E_X^\Phi\left[ \left(\sum_{\mathfrak p \in P} \frac{\chi_\mathfrak q(\mathfrak p)}{N(\mathfrak p)^{1/2}} + \sum_{\mathfrak p \in P} \frac{\overline{\chi_\mathfrak q(\mathfrak p)}}{N(\mathfrak p)^{1/2}} \right)^k \right]
.\]
Since $\chi_\mathfrak q$ is an $n$-th order character, $\overline{\chi_\mathfrak q(\mathfrak p)}=\chi_\mathfrak q(\mathfrak p^{n-1})$. After applying the binomial theorem we find that
\[
\E_X^\Phi\left[(\Re \mathcal P(\mathfrak q))^k\right] = \frac{1}{2^k} \E_X^\Phi\left[\sum_{i=0}^k {k \choose i} \left( \sum_{\mathfrak p \in P} \frac{\chi_\mathfrak q(\mathfrak p)}{N(\mathfrak p)^{1/2}}\right)^{k-i}\left( \sum_{\mathfrak p \in P} \frac{\chi_\mathfrak q(\mathfrak p^{n-1})}{N(\mathfrak p)^{1/2}}\right)^{i}\right].
\]
We can rewrite this as a Dirichlet sum, over ideals $\mathfrak m, \mathfrak n$ such that $\Omega(\mathfrak m) = k-i, \Omega(\mathfrak n)=i$, where $\Omega$ is the number of prime factors counting with multiplicity, and every prime factor is in $P$. We introduce a multiplicative function $w$ by defining $w(\mathfrak p^k)=k!$. Then after handling the combinatorics and interchanging the order of summation we arrive at the following expression:
\begin{align}
\label{expanded1}
\E_X^\Phi\left[(\Re \mathcal P(\mathfrak q))^k \right] = \frac{k!}{2^k} \sum_{\substack{\mathfrak m, \mathfrak n  \\  \Omega(\mathfrak m \mathfrak n)=k}} \frac{1}{N(\mathfrak m \mathfrak n)^{1/2} w(\mathfrak m)w(\mathfrak n)} \E_X^\Phi\left[ \chi_\mathfrak q(\mathfrak m \mathfrak n^{n-1}) \right].
\end{align}
By convention, throughout this paper, we shall omit specifying that $\mathfrak m, \mathfrak n$ are supported on a suitable set of primes $P$, and write as above.

To deal with the expected value $\E_X^\Phi\left[ \chi_\mathfrak q(\mathfrak m \mathfrak n^{n-1})\right]$ we prove the following lemma.
\begin{lem}
\label{expected_character}
Let $\mathfrak n$ be an integral ideal coprime to $S$, and denote by $\rad \mathfrak n$ the product of primes dividing $\mathfrak n$. If $\mathfrak n$ is an $n$-th power, then 
\[
\E_X^\Phi\left[ \chi_\mathfrak q(\mathfrak n) \right] = \prod_{\mathfrak p \mid \mathfrak n} \left( 1 + \frac{1}{N(\mathfrak p)}\right)^{-1} +O(X^{-1/2+\varepsilon}N(\rad \mathfrak n)^{1/4+\varepsilon}).
\]
If $\mathfrak n$ is not an $n$-th power, then
\[
\E_X^\Phi[\chi_\mathfrak q(\mathfrak n)] \ll X^{-1/2+\varepsilon}N(\rad \mathfrak n)^{1/4+\varepsilon}.
\]
\begin{proof}
If $\mathfrak n$ is an $n$-th power, then $\chi_\mathfrak q(\mathfrak n)=1$ for $(\mathfrak n, \mathfrak q)=1$ and we proceed as in Lemma \ref{familysizelem}. Otherwise, we write $\mathfrak n = \mathfrak a \mathfrak b^{n}$ where $\mathfrak a$ is $n$-th power free. Then applying Lemma \ref{reciprocitylem} gives
\[
\sum_{\mathfrak q \in \mathcal F_n} \chi_\mathfrak q(\mathfrak n) \Phi \left( \frac{N(\mathfrak q)}{X}\right) = \sum_{\mathfrak q \in \mathcal F_n} \chi_0^\mathfrak b(\mathfrak q)\chi_\mathfrak a(\mathfrak q) (m_G,m_E)_S \Phi \left( \frac{N(\mathfrak q)}{X}\right),
\]
where $\chi_0^\mathfrak b$ is the principal character of modulus $\mathfrak b$.

Let $h(\mathfrak q) := (m_G,m_E)_S$, where $E$ depends on $\mathfrak q$. The function $h : H_\mathfrak c \to \C$ may not be multiplicative. However, we can write it as a linear combination of characters on $H_\mathfrak c$. By Fourier inversion,
\[
h(\mathfrak q) = \frac{1}{\sqrt{|H_\mathfrak c|}} \sum_{\chi \in H_\mathfrak c^*} \overline{\chi(\mathfrak q)} \widehat{h}(\chi),
\]
where
\[
\widehat{h}(\chi) = \frac{1}{\sqrt{|H_\mathfrak c|}} \sum_{I \in H_\mathfrak c} \chi(I)h(I).
\]
Given $\chi \in H_\mathfrak c^*$ we define the Hecke character $\psi_\chi(\mathfrak q):= \chi_0^S(\mathfrak q) \chi_0^\mathfrak b(\mathfrak q)\chi_\mathfrak a(\mathfrak q) \overline{\chi(\mathfrak q)}$, so that
\[
\sum_{\mathfrak q \in \mathcal F_n} \chi_\mathfrak q(\mathfrak n) \Phi \left( \frac{N(\mathfrak q)}{X}\right) = \frac{1}{\sqrt{|H_\mathfrak c|}}\sum_{\chi \in H_\mathfrak c^*} \widehat{h}(\chi) \sum_{\mathfrak q \in \mathcal F_n} \psi_\chi(\mathfrak q) \Phi \left( \frac{N(\mathfrak q)}{X}\right).
\]

Mellin inversion then yields
\[
\sum_{\mathfrak q \in \mathcal F_n} \chi_\mathfrak q(\mathfrak n) \Phi\left( \frac{N(\mathfrak q)}{X}\right) \ll \sum_{\chi \in H_\mathfrak c^*} \left|\frac{1}{2\pi i} \int_{(2)} \widehat{\Phi}(s)D_\chi(s) X^sds\right|,
\]
where
\[
D_\chi(s) :=  \prod_\mathfrak p \left( 1 + \frac{\psi_\chi(\mathfrak p)}{N(\mathfrak p)^s}\right)=\frac{ L(s,\psi_\chi)}{L(2s, \psi_\chi^2)}.
\]
Since $\mathfrak n$ is not an $n$-th  power, there is some prime $\mathfrak p \not\in S$ dividing the conductor of $\chi_\mathfrak a$, note that $\mathfrak p \nmid \mathfrak c$. Consequently, $\psi_\chi$ is non-principal and so $L(s,\psi_\chi)$ is holomorphic for $\Re s>\frac 12$. Since the norm of the conductor of $\chi_\mathfrak a$ is $\ll N(\rad \mathfrak a)$ and the removed Euler factors corresponding to $\mathfrak b$ are $\ll N(\rad \mathfrak b)^\varepsilon$, standard convexity bounds give $L(s,\psi_\chi) \ll N(\rad \mathfrak n)^{1/4+\varepsilon}(1+|t|)^{[K:\Q]/4+\varepsilon}$ bounding the numerator of $D_\chi(s)$. Also, $|L(2s,\psi_\chi^2)| \gg_\varepsilon 1$ bounding the denominator of $D_\chi(s)$. With the rapid decay of $\widehat{\Phi}(s)$ on vertical lines, we shift the contour to $\Re s = \tfrac 12+\varepsilon$, and are left with an error term of $O(N(\rad \mathfrak n)^{1/4+\varepsilon} X^{1/2+\varepsilon})$. Finally dividing by Lemma \ref{familysizelem} gives the stated error term in Lemma \ref{expected_character}.
\end{proof}
\end{lem}
By definition, each prime factor of $\mathfrak m, \mathfrak n$ is bounded by $X^{1/(\log\log X)^2}$, and there are precisely $k$ such factors. Hence, $N(\mathfrak m \mathfrak n) \leq X^{k/(\log\log X)^2}$, which by choosing $X$ large enough makes \eqref{expanded1} a short Dirichlet polynomial of length an arbitrarily small power of $X$. Therefore, Lemma \ref{expected_character} shows that the contribution to (\ref{expanded1}) from the terms with $\mathfrak m \mathfrak n^{n-1}$ a non-$n$-th power or the error terms from $\mathfrak m \mathfrak n^{n-1}$ an $n$-th power are negligible, of size $o(1)$.

We now deal with the contribution from the main terms where $\mathfrak m \mathfrak n^{n-1}$ is an $n$-th power. First we show that the off-diagonal terms are also negligible. For such a term there must be strictly fewer than $k/2$ distinct prime factors of $\mathfrak m \mathfrak n^{n-1}$. It follows that the contribution from off-diagonal $n$-th power terms is 
\[
\ll \left( \sum_{\mathfrak p \in P} \frac{1}{N(\mathfrak p)} \right)^{\lfloor k/2-1/2\rfloor}.
\]
This is $o\left((\log\log X)^{k/2}\right)$ and so is absorbed into the error term of Proposition \ref{pmoments}. 

If $k$ is odd, then there are no terms on the diagonal, and so Proposition \ref{pmoments} follows in this case. It remains to show that when $k$ is even, the diagonal terms give $M_k+o(1)$. Indeed, from (\ref{expanded1}) we find
\[
\E_X^\Phi\left[ (\Re \mathcal P(\mathfrak q))^k \right] = \frac{k!}{2^k(k/2)!}\left( \sum_{\mathfrak p \in P} \frac{1}{N(\mathfrak p)+1}\right)^{k/2}+o\left((\log\log X)^{k/2}\right).
\]
Finally, applying a suitable Mertens' theorem gives
\[
\E_X^\Phi\left[ (\Re \mathcal P(\mathfrak q))^k \right] = \frac{k!}{2^{k/2}(k/2)!} \left(\frac 12 \log\log X\right)^{k/2} + o\left((\log\log X)^{k/2}\right),
\]
where $M_k = \frac{k!}{2^{k/2}(k/2)!}$ is the $k$-th moment of a standard normal distribution.
\end{proof}
\section{Proof of Proposition \ref{kequals-1}}
\begin{proof}
We proceed similarly to the proof of Proposition \ref{pmoments}, expanding with the definition of $\mathcal P(\mathfrak q)$ to obtain
\begin{align}
\label{initialtwistedbound}
\E_X^\Phi \left[ |L(\tfrac 12, \chi_\mathfrak q)|^2 E_\ell(-2 \Re \mathcal P(\mathfrak q)) \right] = \sum_{\substack{\mathfrak m, \mathfrak n \\ \Omega(\mathfrak m \mathfrak n)\leq \ell}} \frac{(-1)^{\Omega(\mathfrak m \mathfrak n)}\E_X^\Phi \left[ |L(\tfrac 12, \chi_\mathfrak q)|^2 \chi_\mathfrak q(\mathfrak m \mathfrak n^{n-1})\right] }{N(\mathfrak m \mathfrak n)^{1/2}w(\mathfrak m)w(\mathfrak n)}.
\end{align}
We would now like to replace the twisted second moment $\E_X^\Phi \left[ |L(\tfrac 12, \chi_\mathfrak q)|^2 \chi_\mathfrak q(\mathfrak m \mathfrak n^{n-1})\right]$ with a suitable estimate. The key input here comes from very recent work of Diaconu, Ion, Paşol, and Popa \cite{diaconu-ion-pasol-popa}, in which they obtain asymptotics for the twisted second moments with a power-saving error term for all $n \geq 3$. This extends a recent result of David, de Faveri, Dunn, and Stucky who obtained such an asymptotic for the case $n=3$. We adapt some notation and specialize Theorem 10.5.2 from their paper into the following Lemma, whose derivation from the cited result is given in Section 8.
\begin{lem}[\cite{diaconu-ion-pasol-popa}, Theorem 10.5.2 adapted]
\label{expected_twistedlem}
Fix $\varepsilon>0$, let $\mathfrak m,\mathfrak n$ be integral ideals coprime to $S$ and set $P = \{ \mathfrak p \not\in S \ : \ L_1 \leq N(\mathfrak p) \leq L_2\}$ for $L_2>L_1 \geq 2$. Then
\begin{multline}
\label{expected_twisted}
\E_X^\Phi \left[ | L(\tfrac 12 , \chi_\mathfrak q)|^2 \chi_\mathfrak q(\mathfrak m \mathfrak n^{n-1}) \right] = D \times A_1(\mathfrak m, \mathfrak n,\tfrac 12) \left( \log X + \frac{\widehat{\Phi}'(1)}{\widehat{\Phi}(1)} + A_2(\mathfrak m, \mathfrak n, \tfrac 12)+K_0\right) \\ + O(X^{\delta+\varepsilon-1} N(\rad\mathfrak m \mathfrak n)^{(n-1/2)(1-\delta)+\varepsilon}),
\end{multline}
where $\delta = \frac{2n+1}{2n+2}$ if $n>3$ and $\delta=5/6$ if $n=3$. The constants $D,K_0$ and functions $A_1$ and $A_2$ are defined in Section 8. We note that $A_1 A_2 = \frac{d}{ds} A_1$, and that for $k \in \R$ and $|\tfrac 12 -s |\ll \frac{1}{\log 2L_2}$:
\begin{multline}
\label{Abound}
\left|\sum_{\mathfrak m, \mathfrak n}  \frac{A_1(\mathfrak m, \mathfrak n,s)k^{\Omega(\mathfrak m \mathfrak n)}}{N(\mathfrak m \mathfrak n)^{1/2} w(\mathfrak m) w(\mathfrak n)}\right|
\text{ and }
\left|\sum_{\mathfrak m, \mathfrak n}  \frac{|A_1(\mathfrak m, \mathfrak n,s)|k^{\Omega(\mathfrak m \mathfrak n)}}{N(\mathfrak m \mathfrak n)^{1/2} w(\mathfrak m) w(\mathfrak n)}\right|
\\
\leq \exp \left((k^2+2k) \sum_{\mathfrak p \in P} \frac{1}{N(\mathfrak p)} + O_k\left( \sum_{\mathfrak p \in P} |\tfrac 12 - s|\frac{\log N(\mathfrak p)}{N(\mathfrak p)} + \frac{1}{N(\mathfrak p)^{3/2}} \right) \right).
\end{multline}
Furthermore, we note that $A_1$ is a non-negative multiplicative function and that $A_2$ is a non-positive additive function. Lastly, we shall often write $A_1(\mathfrak m, \mathfrak n)$ and $A_2(\mathfrak m, \mathfrak n)$ to mean $A_1(\mathfrak m, \mathfrak n, \tfrac 12)$ and $A_2(\mathfrak m, \mathfrak n, \tfrac 12)$ respectively.
\end{lem}
We take $L_1=2$ and $L_2 = X^{1/(\log\log X)^2}$. Before we can insert Lemma \ref{expected_twistedlem} into \eqref{initialtwistedbound} we need to show that truncating the number of prime factors of $\mathfrak m \mathfrak n$ to $\ell$ contributes negligible error. Noting non-negativity of $A_1$, Rankin's trick gives
\[
\sum_{\substack{\mathfrak m, \mathfrak n \\ \Omega(\mathfrak m \mathfrak n) > \ell}}  A_1(\mathfrak m, \mathfrak n) \frac{(-1)^{\Omega(\mathfrak m \mathfrak n)}}{N(\mathfrak m \mathfrak n)^{1/2} w(\mathfrak m) w(\mathfrak n)} \ll e^{-\ell} \sum_{\mathfrak m, \mathfrak n} A_1(\mathfrak m, \mathfrak n)\frac{e^{\Omega(\mathfrak m \mathfrak n)}}{N(\mathfrak m \mathfrak n)^{1/2} w(\mathfrak m) w(\mathfrak n)},
\]
which by \eqref{Abound} at $s=\tfrac 12$ and Mertens' second theorem is
\[
\ll \frac{e^{-\ell}}{(\log\log X)^{2e^2+4e}} (\log X)^{e^2+2e}.
\]
Recall that $\ell = 20 \lceil \log\log X \rceil$, and since $20 >e^2+2e+1$ then the above is $o(1/\log X)$. Now we bound the tail from $A_1A_2$, Rankin's trick gives:
\[
\left|\sum_{\substack{\mathfrak m, \mathfrak n \\ \Omega(\mathfrak m \mathfrak n) > \ell}}   \frac{A_1A_2(\mathfrak m, \mathfrak n)(-1)^{\Omega(\mathfrak m \mathfrak n)}}{N(\mathfrak m \mathfrak n)^{1/2} w(\mathfrak m) w(\mathfrak n)}\right| \ll e^{-\ell} \sum_{\mathfrak m, \mathfrak n} \frac{|A_1A_2(\mathfrak m, \mathfrak n)| e^{\Omega(\mathfrak m \mathfrak n)}}{N(\mathfrak m \mathfrak n)^{1/2} w(\mathfrak m) w(\mathfrak n)}.
\]
Since $|A_1A_2|=-A_1A_2$, we can apply Cauchy's estimate on the disk $|\tfrac 12 - s| \ll \frac{1}{\log L_2}$ to \eqref{Abound} to find that the above is
\[
\ll \frac{  e^{-\ell}\log L_2}{(\log\log X)^{2e^2+4e}} (\log X)^{e^2+2e},
\]
which as before will be $o(1/\log X)$.

It only remains to show that summing over the error term from Lemma \ref{expected_twistedlem} gives a negligible contribution. Indeed, \eqref{initialtwistedbound} is a short Dirichlet polynomial of length $\ll X^{O(1/\log\log X)}$, and $N(\rad \mathfrak m \mathfrak n) \leq X^{O(1/\log\log X)}$, so each individual error term is bounded by
\[
\ll X^{\delta+\varepsilon-1+n o(1)}.
\]
Summing over $\mathfrak m, \mathfrak n$ then gives a total contribution bounded by
\[
X^{\delta+\varepsilon-1+o(1)}\sum_{\mathfrak m, \mathfrak n} \frac{1}{w(\mathfrak m)w(\mathfrak n)} \ll X^{\delta+\varepsilon-1+o(1)}.
\]
Evidently this becomes negligible as $X \to \infty$.

Now, inserting Lemma \ref{expected_twistedlem} into \eqref{initialtwistedbound} and using Mertens' theorem yields
\[
\E_X^\Phi \left[ |L(\tfrac 12, \chi_\mathfrak q)|^2 E_\ell(-2 \Re \mathcal P(\mathfrak q)) \right] \ll \left( \log X + \frac{\widehat{\Phi}'(1)}{\widehat{\Phi}(1)}+ 1+ K_0 \right) \frac{(\log\log X)^2}{\log X}+O(1).
\]
The constants $D$ and $K_0$ only depend on the fixed data of choice of number field and prime ideals $S$. Thus the above is $\ll (\log\log X)^2$ as desired. 
\end{proof}
\section{Proof of Proposition \ref{bound1}}
\begin{proof}
We split the proof of Proposition \ref{bound1} into two parts. It suffices to show that
\[
E_1:=\E_X^\Phi\!\left[ \prod_{j=1}^R \mathcal B_j(\mathfrak q) \right]
\ll (\log X)^{k^2},
\]
and
\[
E_2:=\E_X^\Phi\!\left[
\sum_{r=0}^{R-1}
\left(
\prod_{j=1}^r \mathcal B_j(\mathfrak q)
\right)
\left(
\frac{2e^2\,\Re \mathcal P_{r+1}(\mathfrak q)}
{\ell_{r+1}}
\right)^{\ell_{r+1}}
\right]
\ll (\log X)^{k^2}.
\]
We first deal with bounding $E_1$. Inserting the definition of $\mathcal B_j$ and expanding as in the proof of Proposition \ref{kequals-1} yields
\[
E_1 = \E_X^\Phi\left[\prod_{j=1}^R \sum_{\substack{\mathfrak m_j,\mathfrak n_j \\ \Omega(\mathfrak m_j \mathfrak n_j) \leq \ell_j}} \frac{k^{\Omega(\mathfrak m_j \mathfrak n_j)} \chi_\mathfrak q(\mathfrak m_j \mathfrak n_j^{n-1})}{N(\mathfrak m_j \mathfrak n_j)^{1/2} w(\mathfrak m_j) w(\mathfrak n_j)}\right],
\]
where every prime factor of $\mathfrak m_j \mathfrak n_j$ lies in $P_j$. Since the sets of primes $P_j$ are disjoint, setting $\mathfrak m = \mathfrak m_1...\mathfrak m_R$, $\mathfrak n = \mathfrak n_1...\mathfrak n_R$, we find
\[
E_1 = \sum_{\substack{\mathfrak m, \mathfrak n \\ \Omega(\mathfrak m_j \mathfrak n_j) \leq \ell_j}} \frac{k^{\Omega(\mathfrak m \mathfrak n)}}{N(\mathfrak m \mathfrak n)^{1/2} w(\mathfrak m) w(\mathfrak n)} \E_X^\Phi[\chi_\mathfrak q(\mathfrak m \mathfrak n^{n-1})].
\]
We note that this is a Dirichlet sum of length at most $(X^{1/\ell_1^2})^{\ell_1}...(X^{1/\ell_R^2})^{\ell_R} = X^{1/\ell_1+...+1/\ell_R}\ll X^{\tfrac{1}{10000n}}$ by our choice of $\ell_j$. We now apply Lemma \ref{expected_character} as in the proof of Proposition \ref{pmoments}. For the same reason, the contribution from terms with $\mathfrak m \mathfrak n^{n-1}$ a non-$n$-th power will be negligible, of size $O(X^{-1/2+\tfrac{1}{40000n}})$. Therefore, we find
\begin{align}
\label{e1}
E_1 \ll \sum_{\mathfrak m\mathfrak n^{n-1}= \cube^n} \frac{k^{\Omega(\mathfrak m \mathfrak n)}}{N(\mathfrak m \mathfrak n)^{1/2} w(\mathfrak m) w(\mathfrak n)} + O(X^{-1/2+\tfrac{1}{40000n}+\tfrac{1}{10000n}}).
\end{align}
Since all the summands are nonnegative, we can extend it to a sum over $\mathfrak m,\mathfrak n$ with arbitrarily many prime factors. Let $x_\mathfrak p = \frac{k}{N(\mathfrak p)^{1/2}}$, then writing (\ref{e1}) as an Euler product gives
\[
E_1 \ll \prod_{j=1}^R \prod_{\mathfrak p \in P_j} \sum_{i=0}^\infty \sum_{\substack{ a+b=i \\ a\equiv b \bmod n}} \frac{x_\mathfrak p^i}{ a! b!}.
\]
Therefore,
\[
E_1\ll \exp\left(\sum_{j=1}^R\sum_{\mathfrak p \in P_j}   \log\left(1+x_\mathfrak p^2 + O(x_\mathfrak p^3)\right)\right),
\]
which can be bounded above by
\[
\exp \left( \sum_{j=1}^R \sum_{\mathfrak p \in P_j} \frac{k^2}{N(\mathfrak p)} + O\left(\frac{1}{N(\mathfrak p)^{3/2}}\right) \right).
\]
Finally, note that the sum over $\frac{1}{N(\mathfrak p)^{3/2}}$ is bounded uniformly, and that the primes in $P_1,...,P_R$ are bounded in norm by $X^{\tfrac{1}{1000n}}$. Therefore, by Mertens' theorem we obtain
\[
E_1 \ll \exp\left( k^2\log\log X\right) = (\log X)^{k^2}.
\]
It remains to bound $E_2$. The argument is similar to the estimate for $E_1$. First, expanding gives
\begin{multline}
\label{e2}
E_2 = \E_X^\Phi\left[\sum_{r=0}^{R-1}\left( \frac{e^2}{\ell_{r+1}}\right)^{\ell_{r+1}} \ell_{r+1}!\prod_{j=1}^r \left( \sum_{\substack{\mathfrak m_j, \mathfrak n_j \\ \Omega(\mathfrak m_j \mathfrak n_j) \leq \ell_j}} \frac{k^{\Omega(\mathfrak m_j \mathfrak n_j)} \chi_\mathfrak q(\mathfrak m_j \mathfrak n_j^{n-1})}{N(\mathfrak m_j\mathfrak n_j)^{1/2}w(\mathfrak m_j) w(\mathfrak n_j)}\right)\right. 
\\
\left. \times \left( \sum_{\substack{\mathfrak m_{r+1}, \mathfrak n_{r+1} \\ \Omega(\mathfrak m_{r+1}\mathfrak n_{r+1} ) = \ell_{r+1}}} \frac{\chi_\mathfrak q(\mathfrak m_{r+1} \mathfrak n_{r+1}^{n-1})}{N(\mathfrak m_{r+1} \mathfrak n_{r+1})^{1/2} w(\mathfrak m_{r+1}) w(\mathfrak n_{r+1})}\right) \right].
\end{multline}
As before, we define $\mathfrak m = \mathfrak m_1...\mathfrak m_{r+1}$ and $\mathfrak n = \mathfrak n_1... \mathfrak n_{r+1}$. Then, using the fact that $P_j$ are disjoint, we find
\[
E_2 = \sum_{r=0}^{R-1} \left( \frac{e^2}{\ell_{r+1}}\right)^{\ell_{r+1}} \ell_{r+1}! \sum_{\mathfrak m, \mathfrak n} \frac{k^{\Omega(\mathfrak m \mathfrak n) - \Omega(\mathfrak m_{r+1} \mathfrak n_{r+1})}}{N(\mathfrak m \mathfrak n)^{1/2} w(\mathfrak m) w(\mathfrak n)} \E_X^\Phi[\chi_\mathfrak q(\mathfrak m \mathfrak n^{n-1})].
\]
Now, by Lemma \ref{expected_character}, any terms with $\mathfrak m \mathfrak n^{n-1}$ a non-$n$-th power will contribute negligibly. Therefore,
\begin{multline}
E_2 \ll \sum_{r=0}^{R-1}\left( \frac{e^2}{\ell_{r+1}}\right)^{\ell_{r+1}} \ell_{r+1}!\prod_{j=1}^r \left( \sum_{\substack{\mathfrak m_j, \mathfrak n_j \\ \Omega(\mathfrak m_j \mathfrak n_j) \leq \ell_j \\ \mathfrak m_j \mathfrak n_j^{n-1} = \cube^n}} \frac{k^{\Omega(\mathfrak m_j \mathfrak n_j)} }{N(\mathfrak m_j\mathfrak n_j)^{1/2}w(\mathfrak m_j) w(\mathfrak n_j)}\right)
\\
\times \left( \sum_{\substack{\mathfrak m_{r+1}, \mathfrak n_{r+1} \\ \Omega(\mathfrak m_{r+1}\mathfrak n_{r+1} ) = \ell_{r+1} \\ \mathfrak m_{r+1} \mathfrak n_{r+1}^{n-1}=\cube^n}} \frac{1}{N(\mathfrak m_{r+1} \mathfrak n_{r+1})^{1/2} w(\mathfrak m_{r+1}) w(\mathfrak n_{r+1})}\right) +O(X^{-1/2+\tfrac{1}{50}}).
\end{multline}
The product of the first bracketed sums above can be bounded, in the same way as $E_1$, by
\[
\prod_{j=1}^r \left( \sum_{\substack{\mathfrak m_j, \mathfrak n_j \\ \Omega(\mathfrak m_j \mathfrak n_j) \leq \ell_j \\ \mathfrak m_j \mathfrak n_j^{n-1} = \cube^n}} \frac{k^{\Omega(\mathfrak m_j \mathfrak n_j)} }{N(\mathfrak m_j\mathfrak n_j)^{1/2}w(\mathfrak m_j) w(\mathfrak n_j)}\right) \ll \exp \left( k^2\sum_{j=1}^r \sum_{\mathfrak p \in P_j} \frac{1}{N(\mathfrak p)}\right).
\]
The second bracketed term can be bounded using Rankin's trick, observing that by positivity,
\begin{multline}
 \sum_{\substack{\mathfrak m_{r+1}, \mathfrak n_{r+1} \\ \Omega(\mathfrak m_{r+1}\mathfrak n_{r+1} ) = \ell_{r+1} \\ \mathfrak m_{r+1} \mathfrak n_{r+1}^{n-1}=\cube^n}} \frac{1}{N(\mathfrak m_{r+1} \mathfrak n_{r+1})^{1/2} w(\mathfrak m_{r+1}) w(\mathfrak n_{r+1})} \\ \ll e^{-2\ell_{r+1}} \times \sum_{\substack{\mathfrak m_{r+1}, \mathfrak n_{r+1} \\ \mathfrak m_{r+1} \mathfrak n_{r+1}^{n-1} =  \cube^n}} \frac{e^{2\Omega(\mathfrak m_{r+1} \mathfrak n_{r+1})}}{N(\mathfrak m_{r+1} \mathfrak n_{r+1})^{1/2} w(\mathfrak m_{r+1})w(\mathfrak n_{r+1})}.
\end{multline}
The above sum can then be bounded as in $E_1$ to obtain
\[
 \sum_{\substack{\mathfrak m_{r+1}, \mathfrak n_{r+1} \\ \Omega(\mathfrak m_{r+1}\mathfrak n_{r+1} ) = \ell_{r+1} \\ \mathfrak m_{r+1} \mathfrak n_{r+1}^{n-1}=\cube^n}} \frac{1}{N(\mathfrak m_{r+1} \mathfrak n_{r+1})^{1/2} w(\mathfrak m_{r+1}) w(\mathfrak n_{r+1})} \ll  \exp \left( -2\ell_{r+1}+e^4\sum_{\mathfrak p \in P_{r+1}} \frac{1}{N(\mathfrak p)}\right).
\]
Since $\sum_{\mathfrak p \in P_{r+1}} \frac{1}{N(\mathfrak p)} \leq \ell_{r+1}/2700$, this is bounded above by $e^{-1.9 \ell_{r+1}}$.

Therefore we find

\[
E_2 \ll  \left( \frac{e^2}{\ell_{1}} \right)^{\ell_{1}} \ell_1! e^{-1.9 \ell_1} +   \sum_{r=1}^{R-1} \left( \frac{e^2}{\ell_{r+1}} \right)^{\ell_{r+1}} \ell_{r+1} ! \exp \left( k^2 \sum_{N(\mathfrak p) \leq X^{1/\ell_r^2}} \frac{1}{N(\mathfrak p)}\right) e^{-1.9\ell_{r+1}}.
\]
Applying Stirling's formula and Mertens' theorem yields
\[
E_2 \ll \sum_{r=0}^{R-1} e^{-\ell_{r+1}/2}(\log X)^{k^2} \ll (\log X)^{k^2}
\]
Combining this bound with the bound for $E_1$ then yields Proposition \ref{bound1}.
\end{proof}
\section{Proof of Proposition \ref{bound2}}
The strategy for proving Proposition \ref{bound2} is largely the same as in the previous section. We reuse the notation, splitting the problem into two parts:
\[
E_1:= \E_X^\Phi \left[|L(\tfrac{1}{2},\chi_\mathfrak q)|^2 \prod_{j=1}^R \mathcal A_j(\mathfrak q) \right]\ll (\log X)^{k^2}
\] 
and
\[
E_2:= \E_X^\Phi \left[ |L(\tfrac{1}{2},\chi_\mathfrak q)|^2 \sum_{r=0}^{R-1} \prod_{j=1}^r \mathcal A_j(\mathfrak q) \left( \frac{2e^2 \Re \mathcal P_{r+1}(\mathfrak q)}{\ell_{r+1}}\right)^{\ell_{r+1}} \right] \ll (\log X)^{k^2}.
\]
As before, we first deal with bounding $E_1$. Inserting the definition of $\mathcal A_j$ and expanding gives
\begin{align}
\label{estimate1}
E_1 = \sum_{\substack{ \mathfrak m, \mathfrak n \\ \Omega(\mathfrak m_j \mathfrak n_j) \leq \ell_j}} \frac{(k-1)^{\Omega(\mathfrak m \mathfrak n)}}{N(\mathfrak m \mathfrak n)^{1/2} w(\mathfrak m) w(\mathfrak n)} \E_X^\Phi\left[ |L(\tfrac{1}{2},\chi_\mathfrak q)|^2 \chi_\mathfrak q(\mathfrak m \mathfrak n^{n-1}) \right],
\end{align}
where $\mathfrak m = \mathfrak m_1...\mathfrak m_R$ and $\mathfrak n = \mathfrak n_1...\mathfrak n_R$. 

We proceed similarly to the proof of Proposition \ref{kequals-1}, but this time we apply Rankin's trick to each set of primes $P_j$ separately. Let $|\tfrac 12-s| \ll \frac{\ell_R^2}{\log X}$, then 
\begin{multline*}
\left|\sum_{\substack{\mathfrak m_j, \mathfrak n_j \\ \Omega(\mathfrak m_j \mathfrak n_j)>\ell_j}} \frac{A_1(\mathfrak m_j,\mathfrak n_j,s)(k-1)^{\Omega(\mathfrak m_j \mathfrak n_j)}}{N(\mathfrak m_j \mathfrak n_j)^{1/2} w(\mathfrak m_j) w(\mathfrak n_j)}\right|
\\
\leq e^{-\ell_j} \sum_{\mathfrak m_j, \mathfrak n_j} |A_1(\mathfrak m_j, \mathfrak n_j,s)| \frac{|ek-e|^{\Omega(\mathfrak m_j \mathfrak n_j)}}{N(\mathfrak m_j \mathfrak n_j)^{1/2} w(\mathfrak m_j) w(\mathfrak n_j)},
\end{multline*}
which by \eqref{Abound} with $L_1=X^{1/\ell_{j-1}^2}$ ($L_1=2$ if $j=1$), $L_2=X^{1/\ell_j^2}$ is bounded by 
\begin{multline*}
\leq \exp \left(O_k\left(\sum_{\mathfrak p \in P_j} |\tfrac 12 - s| \frac{\log N(\mathfrak p)}{N(\mathfrak p)}+ \frac{1}{N(\mathfrak p)^{3/2}}\right) \right)
\\
\times e^{-\ell_j} \exp \left( ((e-ek)^2+2(e-ek)) \sum_{\mathfrak p \in P_j} \frac{1}{N(\mathfrak p)}\right)
\\
\leq \exp \left(O_k\left(\sum_{\mathfrak p \in P_j}|\tfrac 12 - s| \frac{\log N(\mathfrak p)}{N(\mathfrak p)}+\frac{1}{N(\mathfrak p)^{3/2}}\right) \right)
\\
\times e^{-\ell_j/2}  \exp \left( (k^2-1) \sum_{\mathfrak p \in P_j} \frac{1}{N(\mathfrak p)}\right) .
\end{multline*}

Consequently, \eqref{Abound} along with multiplicativity of $A_1$ yields the following bound:
\begin{multline}
\label{truncatedAbound}
\left|\sum_{\substack{\mathfrak m, \mathfrak n \\ \Omega(\mathfrak m_j \mathfrak n_j)\leq \ell_j}} \frac{A_1(\mathfrak m,\mathfrak n,s)(k-1)^{\Omega(\mathfrak m \mathfrak n)}}{N(\mathfrak m \mathfrak n)^{1/2} w(\mathfrak m) w(\mathfrak n)}\right|
\\
\leq \exp \left(O\left(\sum_{N(\mathfrak p) \leq X^{1/\ell_R^2}}|\tfrac 12 - s|\frac{\log N(\mathfrak p)}{N(\mathfrak p)}+\frac{1}{N(\mathfrak p)^{3/2}}\right) \right) \\
\times \prod_{j=1}^R (1+O(e^{-\ell_j/2})) \exp \left( (k^2-1) \sum_{N(\mathfrak p) \leq X^{1/\ell_R^2}} \frac{1}{N(\mathfrak p)}\right).
\end{multline}
For $s=\tfrac 12$, the error term at the front is $\ll 1$, and so is the product over $j$. Thus by Mertens' theorem we find
\[
\sum_{\substack{\mathfrak m, \mathfrak n \\ \Omega(\mathfrak m_j \mathfrak n_j)\leq \ell_j}} \frac{A_1(\mathfrak m,\mathfrak n)(k-1)^{\Omega(\mathfrak m \mathfrak n)}}{N(\mathfrak m \mathfrak n)^{1/2} w(\mathfrak m) w(\mathfrak n)} \ll \frac{(\log X)^{k^2-1}}{\ell_R^{2k^2-2}} \ll (\log X)^{k^2-1}.
\]
We now deal with the contribution from the $A_1A_2$ terms. By Cauchy's estimate applied to \eqref{truncatedAbound} in the disk $|\tfrac 12 -s| \ll \frac{\ell_R^2}{\log X}$ we obtain
\begin{multline}
\left|\sum_{\substack{\mathfrak m, \mathfrak n \\ \Omega(\mathfrak m_j \mathfrak n_j)\leq \ell_j}} \frac{A_1A_2(\mathfrak m,\mathfrak n)(k-1)^{\Omega(\mathfrak m \mathfrak n)}}{N(\mathfrak m \mathfrak n)^{1/2} w(\mathfrak m) w(\mathfrak n)}\right|
\\
\leq \frac{\log X}{\ell_R^2}\exp \left(O\left(\sum_{N(\mathfrak p) \leq X^{1/\ell_R^2}}\frac{\log N(\mathfrak p)\ell_R^2}{N(\mathfrak p)\log X}+\frac{1}{N(\mathfrak p)^{3/2}}\right) \right) \\
\times \prod_{j=1}^R (1+O(e^{-\ell_j/2})) \exp \left( (k^2-1) \sum_{N(\mathfrak p) \leq X^{1/\ell_R^2}} \frac{1}{N(\mathfrak p)}\right).
\end{multline}
By Mertens' theorem the error term from the sum over $\mathfrak p$ is $O(1)$, so we find
\[
\sum_{\substack{\mathfrak m, \mathfrak n \\ \Omega(\mathfrak m_j \mathfrak n_j)\leq \ell_j}} \frac{A_1A_2(\mathfrak m,\mathfrak n)(k-1)^{\Omega(\mathfrak m \mathfrak n)}}{N(\mathfrak m \mathfrak n)^{1/2} w(\mathfrak m) w(\mathfrak n)} \ll (\log X)^{k^2}.
\]
We note that summing over the error term from Lemma \ref{expected_twistedlem} contributes negligibly, by the same argument as in the previous section. Consequently, inserting Lemma \ref{expected_twistedlem} into \eqref{estimate1}, and using the bounds just proven, we deduce that
\[
E_1 \ll (\log X)^{k^2}.
\]
We now proceed to bound $E_2$. Analogously to the argument in Section 6, we find
\[
E_2 = \sum_{r=0}^{R-1} \left( \frac{e^2}{\ell_{r+1}}\right)^{\ell_{r+1}} \ell_{r+1}! \sum_{\mathfrak m, \mathfrak n} \frac{(k-1)^{\Omega(\mathfrak m \mathfrak n)-\Omega(\mathfrak m_{r+1}\mathfrak n_{r+1})}}{N(\mathfrak m\mathfrak n)^{1/2} w(\mathfrak m)w(\mathfrak n)} \E_X^\Phi\left[ |L(\tfrac{1}{2},\chi_\mathfrak q)|^2 \chi_\mathfrak q(\mathfrak m \mathfrak n^{n-1}) \right],
\]
where $\mathfrak m=\mathfrak m_1 \dots \mathfrak m_{r+1}$ and $\mathfrak n = \mathfrak n_1\dots \mathfrak n_{r+1}$, with $\Omega(\mathfrak m_{r+1} \mathfrak n_{r+1})=\ell_{r+1}$ and $\Omega(\mathfrak m_j \mathfrak n_j) \leq \ell_j$ otherwise. We now insert Lemma \ref{expected_twistedlem}. As in the bounding of $E_1$ we first deal with the contribution from terms with $A_1$. We need to bound
\[
E_{11}:=\sum_{r=0}^{R-1} \left( \frac{e^2}{\ell_{r+1}}\right)^{\ell_{r+1}} \ell_{r+1}! \sum_{\mathfrak m, \mathfrak n} A_1(\mathfrak m, \mathfrak n) \frac{(k-1)^{\Omega(\mathfrak m \mathfrak n)-\Omega(\mathfrak m_{r+1}\mathfrak n_{r+1})}}{N(\mathfrak m\mathfrak n)^{1/2} w(\mathfrak m)w(\mathfrak n)}.
\]
By multiplicativity, we write this as
\begin{multline*}
E_{11} = \sum_{r=0}^{R-1} \left( \frac{e^2}{\ell_{r+1}}\right)^{\ell_{r+1}} \ell_{r+1}! \prod_{j=1}^r \left(\sum_{\substack{\mathfrak m_j,\mathfrak n_j \\ \Omega(\mathfrak m_j\mathfrak n_j) \leq \ell_j}} \frac{A_1(\mathfrak m_j, \mathfrak n_j)(k-1)^{\Omega(\mathfrak m_j \mathfrak n_j)}}{N(\mathfrak m_j \mathfrak n_j)^{1/2} w(\mathfrak m_j)w(\mathfrak n_j)} \right)
\\
\times \left( \sum_{\substack{\mathfrak m_{r+1},\mathfrak n_{r+1} \\ \Omega(\mathfrak m_{r+1} \mathfrak n_{r+1}) = \ell_{r+1}}} \frac{A_1(\mathfrak m_{r+1},\mathfrak n_{r+1})}{N(\mathfrak m_{r+1} \mathfrak n_{r+1})^{1/2} w(\mathfrak m_{r+1})w(\mathfrak n_{r+1})} \right).
\end{multline*}
We already bounded the first bracketed term in bounding $E_1$ (see e.g. \eqref{truncatedAbound}). The second bracketed term can be bounded by noting that by \eqref{Abound} it is 
\[
\ll e^{-2\ell_{r+1}} \exp \left( (e^4+2e^2) \sum_{\mathfrak p \in P_{r+1}} \frac{1}{N(\mathfrak p)}\right) \ll e^{-2\ell_{r+1} + (e^4+2e^2)\ell_{r+1}/2700} \ll e^{-1.9 \ell_{r+1}}.
\]
Therefore,
\begin{multline*}
E_{11} \ll \left( \frac{e^2}{\ell_{1}}\right)^{\ell_1} \ell_1! e^{-1.9 \ell_1} 
\\
+  \sum_{r=1}^{R-1}  \left( \frac{e^2}{\ell_{r+1}}\right)^{\ell_{r+1}} \ell_{r+1}! \exp \left( (k^2-1) \sum_{N(\mathfrak p) \leq X^{1/\ell_r^2}} \frac{1}{N(\mathfrak p)}\right) e^{-1.9\ell_{r+1}}.
\end{multline*}
By Stirling's formula this is
\[
\ll (\log X)^{k^2-1} \sum_{r=0}^{R-1} e^{-\ell_{r+1}/2} \ll (\log X)^{k^2-1}.
\]
We also need a corresponding bound for the $A_1A_2$ terms. Indeed, write the above bound for $E_{11}$, but replace $A_1(\mathfrak m, \mathfrak n)$ with $A_1(\mathfrak m, \mathfrak n, s)$, and take absolute values when bounding the second bracketed term. We apply Cauchy's estimate as in bounding $E_1$, which gives us an upper bound of $(\log X)^{k^2}$.
\section{Proof of Lemma \ref{expected_twistedlem}}
The only remaining step is to prove Lemma \ref{expected_twistedlem}. As stated in the Lemma, this is mainly an adaptation of Proposition 10.4.4 and Theorem 10.5.2 from \cite{diaconu-ion-pasol-popa}. We define
\[
\mathfrak a := \prod_{\substack{\mathfrak p \mid \mathfrak m \mathfrak n \\ n \nmid v_\mathfrak p(\mathfrak m \mathfrak n^{n-1})}} \mathfrak p^{v_\mathfrak p(\mathfrak m \mathfrak n^{n-1}) \bmod n} \text{ and } \mathfrak b := \prod_{\substack{\mathfrak p \mid \mathfrak m \mathfrak n \\ n \mid v_\mathfrak p(\mathfrak m \mathfrak n^{n-1})}} \mathfrak p,
\]
and set
\[
D:=\zeta_K^S (\tfrac{n}{2}) \Res_{s=1} \zeta_K^S(s) \prod_{\mathfrak p \not\in S} \frac{P_n(\mathfrak p,\tfrac 12)}{1+\frac{1}{N(\mathfrak p)}} \text{ and } K_0 := \frac{\mathcal A_1'(\tfrac 12)}{\mathcal A_1(\tfrac 12)}-\frac{A'(\tfrac 12)}{2}.
\]
(here $K_0$ is defined using the notation of \cite{diaconu-ion-pasol-popa})

Also, let
\[
A_1(\mathfrak m, \mathfrak n,s) := \frac{1}{N(\mathfrak a)^{s}} \prod_{\mathfrak p \mid \mathfrak a}  \frac{1+N(\mathfrak p)^{-(n-2v_\mathfrak p(\mathfrak a))s}}{P_n(\mathfrak p,s)} \prod_{\mathfrak p \mid \mathfrak b} \frac{1+N(\mathfrak p)^{-n s}}{P_n(\mathfrak p,s)},
\]
where
\[
P_n(\mathfrak p,s) := 1+ N(\mathfrak p)^{-1} - N(\mathfrak p)^{-1-2s} + N(\mathfrak p)^{-n s}(1 - N(\mathfrak p)^{-1} + N(\mathfrak p)^{-1-2s}). 
\]
Then we write
\[
A_2(\mathfrak m, \mathfrak n, s) := \frac{A_1'(\mathfrak m, \mathfrak n, s)}{A_1(\mathfrak m, \mathfrak n, s)},
\]
and set
\[
A_1(\mathfrak m, \mathfrak n) := A_1(\mathfrak m, \mathfrak n, \tfrac 12), \ \  A_2(\mathfrak m, \mathfrak n) := A_2(\mathfrak m, \mathfrak n, \tfrac 12).
\]
For $(\mathfrak m_1 \mathfrak n_1,\mathfrak m_2\mathfrak n_2)=1$, we have
\[
A_1(\mathfrak m_1 \mathfrak m_2, \mathfrak n_1 \mathfrak n_2) = A_1(\mathfrak m_1, \mathfrak n_1) A_1(\mathfrak m_2, \mathfrak n_2)
\]
and
\[
A_2(\mathfrak m_1 \mathfrak m_2, \mathfrak n_1 \mathfrak n_2) = A_2(\mathfrak m_1, \mathfrak n_1) + A_2(\mathfrak m_2, \mathfrak n_2).
\]
Furthermore, $A_1(\mathfrak m, \mathfrak n) \geq 0$ and $A_2(\mathfrak m, \mathfrak n) \leq 0$.

The asymptotic \eqref{expected_twisted} of Lemma \ref{expected_twistedlem} is just Theorem 10.5.2 restated in the notation of this paper, obtained from Proposition 10.4.4. We note that we have normalized by dividing by the family size \eqref{familysize}, and that the function $A_1(\mathfrak m, \mathfrak n, s) = \frac{\mathcal A_1^{(\mathfrak a \mathfrak b \overline{\mathfrak b})^*}(s)}{\mathcal A_1(s)}$ in the notation of \cite{diaconu-ion-pasol-popa}.

It remains to show the bound \eqref{Abound}. 
\begin{lem}
Assume the notation of Lemma \ref{expected_twistedlem} and let
\begin{align}
T:=\sum_{\mathfrak m, \mathfrak n} \frac{A_1(\mathfrak m, \mathfrak n, s) k^{\Omega(\mathfrak m \mathfrak n)}}{N(\mathfrak m \mathfrak n)^{1/2} w(\mathfrak m) w(\mathfrak n)}, \ \ T' := \sum_{\mathfrak m, \mathfrak n} \frac{|A_1(\mathfrak m, \mathfrak n, s)| k^{\Omega(\mathfrak m \mathfrak n)}}{N(\mathfrak m \mathfrak n)^{1/2} w(\mathfrak m) w(\mathfrak n)}.
\end{align}
Then uniformly for $|\tfrac 12 - s| \ll \frac{1}{\log 2L_2}$ and small enough implied constant, both $|T|$ and $|T'|$ are bounded above by
\[
\exp \left((k^2+2k) \sum_{\mathfrak p \in P} \frac{1}{N(\mathfrak p)} +  O_k\left(\sum_{\mathfrak p \in P} |\tfrac 12 -s| \frac{\log N(\mathfrak p)}{N(\mathfrak p)} + \frac{1}{N(\mathfrak p)^{3/2}}\right)\right).
\]
\end{lem}
\begin{proof}
Let us denote
\[
h(\mathfrak a,s) := \frac{1}{N(\mathfrak a)^{s}}\prod_{\mathfrak p \mid \mathfrak a} \frac{1+N(\mathfrak p)^{-(n-2v_\mathfrak p(\mathfrak a))s}}{P_n(\mathfrak p,s)} \text{ and } g(\mathfrak b,s) := \prod_{\mathfrak p \mid \mathfrak b} \frac{1+N(\mathfrak p)^{-n s}}{P_n(\mathfrak p,s)},
\]
so that $A_1(\mathfrak m, \mathfrak n, s) = g(\mathfrak b,s)h(\mathfrak a,s)$. Also, note the reflection formula $h(\mathfrak p^{n-j},s) = h(\mathfrak p^j,s)$. We have the following estimates, uniform in $s$:
\[
g(\mathfrak p,s) = 1 +O\left(\frac{1}{N(\mathfrak p)}\right)
\]
and for $j < \tfrac n2$
\[
h(\mathfrak p^j, s) = N(\mathfrak p)^{-j/2} + O(j |\tfrac 12-s|\log N(\mathfrak p)N(\mathfrak p)^{-j/2} + N(\mathfrak p)^{-1})
\]
and finally for $j = n/2$: 
\[
h(\mathfrak p^j,s) = 2N(\mathfrak p)^{-j/2}+O(N(\mathfrak p)^{-1}).
\]
Factoring over primes in $P$ gives $T=\prod_{\mathfrak p \in P} T_\mathfrak p$, where
\[
T_\mathfrak p =  1+g(\mathfrak p,s) \sum_{i=2}^\infty \sum_{\substack{a+b=i \\ a \equiv b \bmod n}} \frac{x_\mathfrak p^i}{a!b!} + \sum_{j=1}^{n-1} h(\mathfrak p^j,s)\sum_{i=1}^\infty \sum_{\substack{a+b=i \\ a-b \equiv j \bmod n}} \frac{x_\mathfrak p^i}{a!b!},
\]
where $x_\mathfrak p = \frac{k}{N(\mathfrak p)^{1/2}}$. Inserting the estimates for $g$ and $h$ yields
\[
T_\mathfrak p = 1+ \frac{k^2+2k}{N(\mathfrak p)} + O_k\left( | \tfrac 12-s| \frac{\log N(\mathfrak p)}{N(\mathfrak p)} + \frac{1}{N(\mathfrak p)^{3/2}}\right), 
\]
where the implied constants are uniform in $\mathfrak p$. Hence,
\[
|T_\mathfrak p| \leq \exp \left(\frac{k^2+2k}{N(\mathfrak p)} +  O_k\left(|\tfrac 12 -s| \frac{\log N(\mathfrak p)}{N(\mathfrak p)} + \frac{1}{N(\mathfrak p)^{3/2}}\right)\right).
\]

Consequently,
\[
|T| \leq \exp \left((k^2+2k) \sum_{\mathfrak p \in P} \frac{1}{N(\mathfrak p)} +  O_k\left(\sum_{\mathfrak p \in P} |\tfrac 12 -s| \frac{\log N(\mathfrak p)}{N(\mathfrak p)} + \frac{1}{N(\mathfrak p)^{3/2}}\right)\right).
\]
The proof for $T'$ proceeds identically, since applying the triangle inequality to $|g|$ and $|h|$ gives the same bounds as for $g$ and $h$.
\end{proof}

\end{document}